\documentclass[12pt]{amsart}
\usepackage{amssymb,amsmath,amsthm,mathrsfs,multirow,xcolor,framed,url}
\usepackage[pdftex,
         pdfauthor={Rong Chen},
         pdftitle={temp},
         pdfsubject={MATHEMATICS},
         pdfkeywords={temp},
         pdfproducer={Latex with hyperref},
         pdfcreator={pdflatex}]{hyperref}
\usepackage [latin1]{inputenc}
\newtheorem{theorem}{Theorem}[section]
\newtheorem{lemma}[theorem]{Lemma}
\newtheorem{cor}[theorem]{Corollary}
\newtheorem{conj}[theorem]{Conjecture}
\newtheorem{prop}[theorem]{Proposition}

\theoremstyle{definition}

\newtheorem{example}[theorem]{Example}

\theoremstyle{remark}
\newtheorem{remark}[theorem]{Remark}

\numberwithin{equation}{section}

\newcommand\nutwid{\overset {\text{\lower 3pt\hbox{$\sim$}}}\nu}

\allowdisplaybreaks
\newcommand\omycite[1]{}

\newcommand{\beqs}{\begin{equation*}}
\newcommand{\eeqs}{\end{equation*}}
\newcommand{\beq}{\begin{equation}}
\newcommand{\eeq}{\end{equation}}
\begin{document}
%%PRELIMINARY VERSION
\title[Transformation Formulae and Applications for Double Lambert Series]{Transformation Formulae and Applications for Double Lambert Series}

\author{Rong Chen}
\address{Department of Mathematics, Shanghai Normal University, People's Republic of China}
\email{rchen@shnu.edu.cn}
\author{Tianjian Xu}
\address{Department of Mathematics, Shanghai University, People's Republic of China}
\email{xtjmath@shu.edu.cn}

\subjclass[2010]{11B65, 05A15}

\keywords{Lambert series, double Lambert series, transformations, divisor function}

\begin{abstract}
In this paper, we study a class of double Lambert series and establish several identities and transformation relations for them. These formulae provide useful tools for reducing certain double Lambert series to single Lambert series. As applications, we derive identities related to recent conjectures of Andrews, Dixit, Schultz, and Yee, and of Amdeberhan, Andrews, and Ballantine. We also propose a new proof of a result of Amdeberhan, Andrews, and Ballantine.
\end{abstract}
\maketitle
\section{Introduction}
In 1771, J. H. Lambert \cite{lam65} showed that the generating function of the divisor function can be expressed as
\begin{align*}
\sum_{n=1}^\infty\frac{q^n}{1-q^n}.
\end{align*}
This representation naturally paved the way for broader generalizations. A classical Lambert series is formally defined as
\begin{align*}
\sum_{n=1}^\infty\frac{a_nq^n}{1-q^n}.
\end{align*}
In particular, when $a_n=n^{2k-1}$, the resulting Lambert series is a modular form for $k\geq2$ and is a quasimodular form when $k=1$. An excellent survey of existing key results and properties of Lambert series identities can be found in the work of Schmidt \cite{sch20}. Grounded in this classical theoretical framework, recent progress has substantially extended the boundaries of Lambert series theory. A landmark contribution comes from Amdeberhan, Andrews, and Ballantine \cite{amanba26}, who developed a far-reaching generalization through the study of generalized Lambert series of the form
\begin{align*}
\sum_{n\geq1}R_n(q^n,q),
\end{align*}
where $R_n(x,y)$ is a rational function of $x$ and $y$. Moreover, they developed a two-parameter extension named the double Lambert series \cite{amanba26}, expressed by
\begin{align*}
\sum_{m,n\geq1}S_{n,m}(q^n,q^m,q),
\end{align*}
where $S_{n,m}(x,y,z)$ is a rational function of $x$, $y$, and $z$. They \cite{amanba26} established several connections between these series and Rogers-Ramanujan type q-series.

Recall the basic notation and definitions used throughout this paper. Assume that $|q|<1$ for convergence. The $q$-shifted factorials are defined as follows \cite{red}:
\begin{align*}
    (a;q)_0:=1,~~(a; q)_n := \prod_{k=0}^{n-1} (1 - a q^k),~~(a; q)_\infty := \prod_{k=0}^{\infty} (1 - a q^k).
\end{align*}
We also adopt the compact product notation:
\begin{align*}
    (a_1, \cdots, a_m; q)_k := (a_1; q)_k \cdots (a_m; q)_k, \quad k \in \mathbb{Z} \cup \{\infty\}.
\end{align*}

In this paper, we primarily focus on the following double Lambert series:
$$
A(x,y,z,w):=A(x,y,z,w;q):=\sum_{n=0}^\infty \sum_{m\geq n}\frac{x^ny^m}{(1-wq^n)(1-zq^m)},
$$
with $z\neq q^{-r}$, $w\neq q^{-r}$, $|xy|<1$, and $|y|<1$ for $r\in \mathbb{N}$. We provide some basic tools for handling this series. The main goal is to represent this type of double series as a single Lambert series,
$$
L(x,y_1,y_2,\dots,y_k):=L(x,y_1,y_2,\dots,y_k;q):=\sum_{n=0}^\infty \frac{x^n}{(1-y_1q^n)(1-y_2q^n)\dots(1-y_kq^n)},
$$
with $y_i\neq q^{-r}$ and $|x|<1$ for $r\in \mathbb{N}$. We also employ the bilateral Lambert series and the celebrated Ramanujan's ${}_1\psi_1$ summation identity \cite[Entry 29]{be91},
\begin{align}
\label{r:psi11}
L^*(x,y):=L^*(x,y;q):=\sum_{n\in \mathbb{Z}}\frac{x^n}{1-yq^n}=\frac{(q,q,xy,q/xy;q)_\infty}{(x,q/x,y,q/y;q)_\infty},
\end{align}
with $x\neq q^{r}$, $y\neq q^{r}$, and $|q|<|x|<1$.

In 2015, Andrews, Dixit, and Yee \cite{ady15} provided a new partition-theoretic interpretation for the coefficients of the classical Ramanujan/Watson mock theta function $\omega(q)$, formally defined by
\begin{align*}
\omega(q):=\sum_{n=0}^\infty\frac{q^{2n^2+2n}}{(q;q^2)_{n+1}^2}.
\end{align*}
Specifically, $p_\omega(n)$, which is the coefficient of $q^n$ in $q\omega(q)$, counts the number of partitions of $n$ in which all odd parts are less than twice the smallest part. Building on this work, Andrews, Dixit, Schultz and Yee \cite{adsy17} generalized the interpretation to overpartitions, a variant of partitions where each distinct part may be overlined exactly once.

This generalization led them \cite{adsy17} to introduce the counting function $\bar{p}_\omega(n)$, which enumerates overpartitions of $n$ with two additional constraints: all odd parts are less than twice the smallest part, and the smallest part is necessarily overlined. Using q-series techniques, they \cite[Theorem 1.4]{adsy17} derived several remarkable congruence properties for this function, most notably that
\begin{align*}
\bar{p}_\omega(4n+3)&\equiv0~(mod~4),\\
\bar{p}_\omega(8n+6)&\equiv0~(mod~4)
\end{align*}
for all non-negative integers $n$. In the final section of their paper, they remarked that these two congruences can be proven alternatively, provided that the conjecture stated below is valid.
\begin{conj}\cite[Problem 2]{adsy17}\label{adsy}
Let
\begin{align*}
Y(q):=\sum_{n,m\geq1}\frac{(-1)^mq^{2mn+m}}{(1+q^n)(1-q^{2m-1})}.
\end{align*}
Then $Y(q)$ is an odd function of $q$.
\end{conj}

Amdeberhan, Andrews, and Ballantine \cite{amanba26} came up with the following two conjectures when they were attempting to prove Conjecture \ref{adsy}.
\begin{conj}\cite[Conjecture 5.12]{amanba26}\label{AAB-conj-I}
Let $a$ be a positive integer. Then, for each positive integer $n$, we have
\begin{align*}
[q^{n2^a}]\sum_{m,n\geq 1}\frac{q^{mn2^a}}{(1+q^{n2^{a-1}})(1-q^{2m-1})}=\sigma_1(n),
\end{align*}
where $\sigma_1(n)$ is the sum of divisors of $n$.
\end{conj}
\begin{conj}\cite[Conjecture 5.13]{amanba26}\label{AAB-conj-II}
If $r$ is a positive integer, then
\begin{align*}
[q^{2r}]\sum_{m,n\geq1}\frac{q^{2mn}}{(1+q^{2n-1})(1-q^{2m-1})}=[q^{2r}]\sum_{n\geq1}\frac{(n-1)q^n}{1+q^{2n-1}}.
\end{align*}
\end{conj}

Recently, Conjecture \ref{adsy} and \ref{AAB-conj-II} have been proved by Cui and Tang \cite{ct26}. Fang \cite{fang26} independently provided a short proof of Conjecture \ref{adsy} earlier. Meanwhile, Kumar and Singh \cite{ks26} proved Conjecture \ref{AAB-conj-I}.

In this paper, we present some transformation relations of $A(x,y,z,w)$ as its fundamental characteristics in Section \ref{trans-rela} and these relations have been employed for special cases in \cite{ct26,fang26,ks26}. We primarily utilize these properties to derive a relation for $A(x,y,z,w)$ via formal manipulations, without considering convergence. Then we derive certain formulae that yield Conjectures \ref{adsy}, \ref{AAB-conj-I}, and \ref{AAB-conj-II}. We prove these formulae using analytic methods, thereby avoiding the repeated application of Propositions \ref{r:prop1} and \ref{r:prop2}. For example, we have (see Example \ref{r:adsy} below)
\begin{align*}
&\sum_{m,n\geq 1}\frac{(-1)^mq^{2mn+m}x^nz^m}{(1+zq^n)(1-q^{2m-1})}\\
=&xzq^2\sum_{k=0}^\infty \frac{x^kq^{k}}{1+zq^{2k+1}}\sum_{k=0}^\infty \frac{x^kq^{2k}}{1+zq^{2k+1}}
-xzq^2\sum_{k=0}^\infty \frac{q^{k}}{1+zq^{2k+1}}\sum_{k=0}^\infty \frac{x^kq^{k}}{1+zq^{k+1}}.
\end{align*}
Setting $x=z=1$ and using \eqref{r:psi11} we have \cite[Theorem 1.2]{ct26}
\begin{align*}
Y(q)=-q\frac{(q^4;q^4)_\infty^4}{(q^2;q^2)_\infty^2}\sum_{k=1}^{\infty}\frac{q^{2k}}{1+q^{2k}},
\end{align*}
which implies Conjecture \ref{adsy}.

The rest of this paper is organized as follows. In Section \ref{trans-rela}, we establish several fundamental transformation formulae for the double Lambert series $A(x,y,z,w)$. In Section \ref{example}, we present more formulae for the double Lambert series $A(x,y,z,w)$. We first derive identities related to the conjectures of Andrews-Dixit-Schultz-Yee and Amdeberhan-Andrews-Ballantine, thereby giving alternative proofs of Conjectures \ref{adsy}, \ref{AAB-conj-I}, and \ref{AAB-conj-II}. We then present a new proof of a result of Amdeberhan, Andrews, and Ballantine.

\section{Double Lambert series}\label{trans-rela}
In this section, we present and prove some transformation relations of $A(x,y,z,w)$ as Proposition \ref{r:prop1} and Proposition \ref{r:prop2}.
\begin{prop}
\label{r:prop1}
We have the following transformation formula,
\beq
\label{r:m1}
A(x,y,z,w)=A(z/w,w,xy,y),
\eeq
with $|xy|<1$, $|y|<1$, $|z|<1$, and $|w|<1$,
\beq
\label{r:m2}
A(x,y,z,w)+A(y,x,w,z)=L(x,w)L(y,z)+L(xy,z,w),
\eeq
with $|x|<1$, $|y|<1$, $z\neq q^{-r}$, and $w\neq q^{-r}$ for $r\in\mathbb{N}$.
\end{prop}
\begin{proof}
Rearranging and using properties of geometric series, we obtain
\begin{align*}
A(x,y,z,w)&=\sum_{n=0}^\infty \sum_{m=0}^\infty\frac{x^ny^{m+n}}{(1-wq^n)(1-zq^{n+m})}\\
&=\sum_{m,n\geq 0}\sum_{i,j\geq 0}x^ny^{m+n}(wq^n)^i(zq^{m+n})^j\\
&=\sum_{i,j\geq 0}\frac{w^iz^j}{(1-xyq^{i+j})(1-yq^j)}=A(z/w,w,xy,y).
\end{align*}
Similarly, we obtain
\begin{align*}
&A(x,y,z,w)+A(y,x,w,z)\\
=&\sum_{n=0}^\infty \sum_{m\geq n}\frac{x^ny^m}{(1-wq^n)(1-zq^m)}+\sum_{n=0}^\infty \sum_{m\geq n}\frac{y^nx^m}{(1-zq^n)(1-wq^m)}\\
=&\sum_{n=0}^\infty \sum_{m\geq n}\frac{x^ny^m}{(1-wq^n)(1-zq^m)}+\sum_{m=0}^\infty \sum_{n=0}^m\frac{y^nx^m}{(1-zq^n)(1-wq^m)}\\
=&\sum_{n=0}^\infty \sum_{m=0}^\infty\frac{x^ny^m}{(1-wq^n)(1-zq^m)}+\sum_{n=0}^\infty\frac{x^ny^n}{(1-wq^n)(1-zq^n)}\\
=&L(x,w)L(y,z)+L(xy,z,w).
\end{align*}
\end{proof}
We now provide an explanation of Proposition \ref{r:prop1}. Let $(x,y,z,w)$ be a quadruple. Define $S(x,y,z,w)=(z/w,w,xy,y)$ and $T(x,y,z,w)=(y,x,w,z)$ according to \eqref{r:m1} and \eqref{r:m2} respectively. It can be verified that $S$ and $T$ generate a group of order $24$ under composition, satisfying $S^2=T^2=(ST)^{12}=I$, where $I$ is the identity transformation. This implies that we can relate $A(x,y,z,w)$ to up to $23$ other series via Proposition \ref{r:prop1}.

We also have the following basic $q$-lifting transformation relations.
\begin{prop}
\label{r:prop2}
The series $A(x,y,z,w)$ satisfies the following identity,
\begin{align}
\nonumber
A(x,y,z,w)=&wA(xq,y,z,w)+L(z,y,xy)\\
\nonumber
=&zA(x,yq,z,w)+\frac{1}{1-y}L(xy,w)\\
\label{r:sfz}
=&yA(x,y,zq,w)+L(xy,z,w)\\
\nonumber
=&xA(x,y,z,wq)+\frac{1}{1-w}L(y,z)-xL(xy,z,wq).
\end{align}
\end{prop}
\begin{proof}
Taking the difference of the double Lambert series on both sides respectively, we obtain
\begin{align*}
&A(x,y,z,w)-wA(xq,y,z,w)\\
=&\sum_{n=0}^\infty \sum_{m\geq n}\frac{x^ny^m}{(1-wq^n)(1-zq^m)}-w\sum_{n=0}^\infty \sum_{m\geq n}\frac{x^nq^ny^m}{(1-wq^n)(1-zq^m)}\\
=&\sum_{n=0}^\infty \sum_{m\geq n}\frac{x^ny^m}{1-zq^m}=\sum_{n=0}^\infty \sum_{m=0}^\infty\frac{x^ny^{m+n}}{1-zq^{m+n}}\\
=&\sum_{m,n,i\geq 0}x^ny^{m+n}(zq^{m+n})^i=\sum_{i=0}^\infty\frac{z^i}{(1-xyq^i)(1-yq^i)}=L(z,y,xy),
\end{align*}
\begin{align*}
&A(x,y,z,w)-zA(x,yq,z,w)\\
=&\sum_{n=0}^\infty \sum_{m\geq n}\frac{x^ny^m}{(1-wq^n)(1-zq^m)}-z\sum_{n=0}^\infty \sum_{m\geq n}\frac{x^nq^my^m}{(1-wq^n)(1-zq^m)}\\
=&\sum_{n=0}^\infty \sum_{m\geq n}\frac{x^ny^m}{(1-wq^n)}=\sum_{n=0}^\infty \sum_{m=0}^\infty\frac{x^ny^{m+n}}{(1-wq^n)}=\frac{1}{1-y}\sum_{n=0}^\infty\frac{x^ny^n}{(1-wq^n)}=\frac{1}{1-y}L(xy,w),
\end{align*}
\begin{align*}
&A(x,y,z,w)-yA(x,y,zq,w)\\
=&\sum_{n=0}^\infty \sum_{m\geq n}\frac{x^ny^m}{(1-wq^n)(1-zq^m)}-\sum_{n=0}^\infty \sum_{m\geq n}\frac{x^ny^{m+1}}{(1-wq^n)(1-zq^{m+1})}\\
=&\sum_{n=0}^\infty \sum_{m\geq n}\frac{x^ny^m}{(1-wq^n)(1-zq^m)}-\sum_{n=0}^\infty \sum_{m\geq n+1}\frac{x^ny^m}{(1-wq^n)(1-zq^m)}\\
=&\sum_{n=0}^\infty\frac{x^ny^n}{(1-wq^n)(1-zq^n)}=L(xy,z,w),
\end{align*}
and
\begin{align*}
&A(x,y,z,w)-xA(x,y,z,wq)\\
=&\sum_{n=0}^\infty \sum_{m\geq n}\frac{x^ny^m}{(1-wq^n)(1-zq^m)}-\sum_{n=0}^\infty \sum_{m\geq n}\frac{x^{n+1}y^m}{(1-wq^{n+1})(1-zq^m)}\\
=&\sum_{n=0}^\infty \sum_{m\geq n}\frac{x^ny^m}{(1-wq^n)(1-zq^m)}-\sum_{n=1}^\infty \sum_{m\geq n-1}\frac{x^ny^m}{(1-wq^n)(1-zq^m)}\\
=&\sum_{n=0}^\infty \sum_{m\geq n}\frac{x^ny^m}{(1-wq^n)(1-zq^m)}-\sum_{n=1}^\infty \sum_{m\geq n}\frac{x^ny^m}{(1-wq^n)(1-zq^m)}-\sum_{n=1}^\infty\frac{x^ny^{n-1}}{(1-wq^n)(1-zq^{n-1})}\\
=&\sum_{m=0}^\infty\frac{y^m}{(1-w)(1-zq^m)}-\sum_{n=0}^\infty\frac{x^{n+1}y^n}{(1-wq^{n+1})(1-zq^n)}\\
=&\frac{1}{1-w}L(y,z)-xL(xy,wq,z).
\end{align*}
\end{proof}
\begin{remark}
Here we do not strictly discuss the convergence for the identities in Proposition \ref{r:prop2}, as we have not actually used this property to complete any proof.
\end{remark}
\section{Examples}\label{example}

\subsection{A useful example of transformation formulae}
We begin with the following useful example and its corollaries.
\begin{example}
For $|x|<1$ and $y\neq q^{-r}$ for $r\in\mathbb{N}$, we obtain
\begin{align}
\label{r:xxyy}
A(x,x,y,y)=\frac{1}{2}L(x,y)^2+\frac{1}{2}L(x^2,y,y).
\end{align}
\end{example}

\begin{proof}
Fix $k\geq0$ and compare the coefficients of $x^k$ on both sides. For the left-hand side,
\begin{align*}
[x^k]A(x,x,y,y)=[x^k]\sum_{n\geq0}\sum_{m\geq n}\frac{x^{m+n}}{(1-yq^m)(1-yq^n)}=\sum_{\substack{n+m=k \\ 0\leq n\leq m}}\frac{1}{(1-yq^m)(1-yq^n)}.
\end{align*}
For the right-hand side,
\begin{align*}
[x^k]L(x,y)^2=[x^k]\left(\sum_{n=0}^\infty \frac{x^n}{1-yq^n}\right)^2=\sum_{n+m=k}\frac{1}{(1-yq^m)(1-yq^n)}
\end{align*}
and
\begin{align*}
[x^k]L(x^2,y,y)=[x^k]\sum_{n=0}^\infty \frac{x^{2n}}{(1-yq^n)^2}=
\begin{cases}
\dfrac{1}{(1-y q^{k/2})^2} & \text{if } k \text{ is even}, \\
0 & \text{if } k \text{ is odd}.
\end{cases}
\end{align*}

When $k$ is odd, then $m\neq n$ for $m+n=k$. By the symmetry of $m$ and $n$, we obtain
\begin{align*}
[x^k]\text{RHS}=\frac{1}{2}\sum_{n+m=k}\frac{1}{(1-y q^n)(1-y q^m)}=\sum_{\substack{n+m=k \\ 0\leq n< m}}\frac{1}{(1-yq^m)(1-yq^n)}=[x^k]\text{LHS}.
\end{align*}
When $k$ is even, let $k=2t$ for some integer $t \geq 0$. At this point, there exists $m=n=t$ for $m+n=k$. By the symmetry of $m$ and $n$, we obtain
\begin{align*}
[x^k]\text{RHS}&=\frac{1}{2}\sum_{n+m=2t}\frac{1}{(1-y q^n)(1-y q^m)}+\frac{1}{2(1-yq^t)^2}\\
&=\sum_{\substack{n+m=2t \\ 0\leq n< m}}\frac{1}{(1-yq^m)(1-yq^n)}+\frac{1}{(1-yq^t)^2}\\
&=\sum_{\substack{n+m=k \\ 0\leq n\leq m}}\frac{1}{(1-yq^m)(1-yq^n)}=[x^k]\text{LHS}.
\end{align*}
\end{proof}

As the applications of \eqref{r:xxyy}, we recover the following identity, which is connected with OEIS A002133.
\begin{cor}
We have
\begin{align*}
\sum_{m=1}^\infty\sum_{n=1}^{m-1}\frac{q^{n+m}}{(1-q^n)(1-q^m)}=\frac{1}{2}(G(q)^2-H(q)),
\end{align*}
where
$$
G(q)=\sum_{k>0}\frac{q^k}{1-q^k},\qquad H(q)=\sum_{k>0}\frac{q^{2k}}{(1-q^k)^2}.
$$
\end{cor}

This follows easily from \eqref{r:xxyy} and \eqref{r:sfz},
\begin{align*}
\sum_{m=1}^\infty\sum_{n=1}^{m-1}\frac{q^{n+m}}{(1-q^n)(1-q^m)}=&q^3A(q,q,q^2,q)\\
=&q^2A(q,q,q,q)-H(q)=\frac{1}{2}(G(q)^2-H(q)).
\end{align*}

Similarly, \cite[Theorem 5.7]{amanba26} provides another example, which was motivated by an attempt to approach Conjecture \ref{adsy}.

\begin{cor}
We have
\begin{align*}
\sum_{m,n\geq 1}\frac{(-q)^{2mn+m-1}}{(1+q^{2n-1})(1-q^{2m-1})}=\sum_{n\geq 1}\frac{q^{4n-2}}{(1-q^{4n-2})^2}.
\end{align*}
\end{cor}

We do not provide a full proof of this equation as it has already been established in \cite{amanba26}. We employ \eqref{r:xxyy} to reduce it to a single sum--an approach that is , in fact, consistent with the one used in \cite{amanba26}.
\begin{align*}
&\sum_{m,n\geq 1}\frac{(-q)^{2mn+m-1}}{(1+q^{2n-1})(1-q^{2m-1})}=q\sum_{n=0}^\infty\frac{q^n}{1+q^{2n+1}}\sum_{j=0}^{n-1}\frac{q^{j}}{1+q^{2j+1}}\\
=&qA(q,q,-q,-q;q^2)-qL(q^2,-q,-q;q^2)=\frac{1}{2}q\left(L(q,-q;q^2)^2-L(q^2,-q,-q;q^2)\right).
\end{align*}

At present, \eqref{r:xxyy} is our only method for directly handling a double Lambert series $A(x,y,z,w)$. However, we shall encounter additional techniques in the examples that follow.

\subsection{The Andrews--Dixit--Schultz--Yee Conjecture}

We consider a generalized form of the Andrews--Dixit--Schultz--Yee conjecture, from which we derive an elegant proof of the conjecture.
\begin{example}
\label{r:adsy}
For $|xq|<1$, $|zq^3|<1$, and $z\neq -q^{-r}$ for $r\in\mathbb{N^*}$, we have
\begin{align*}
&\sum_{m,n\geq 1}\frac{(-1)^mq^{2mn+m}x^nz^m}{(1+zq^n)(1-q^{2m-1})}\\
=&xzq^2\sum_{k=0}^\infty \frac{x^kq^{k}}{1+zq^{2k+1}}\sum_{k=0}^\infty \frac{x^kq^{2k}}{1+zq^{2k+1}}
-xzq^2\sum_{k=0}^\infty \frac{q^{k}}{1+zq^{2k+1}}\sum_{k=0}^\infty \frac{x^kq^{k}}{1+zq^{k+1}}.
\end{align*}
\end{example}

\begin{proof}
Fix $N\geq1$ and compare the coefficients of $x^N$ on both sides. For the left-hand side,
\begin{align*}
[x^N]\sum_{m,n\geq 1}\frac{(-1)^mq^{2mn+m}x^nz^m}{(1+zq^n)(1-q^{2m-1})}=\frac{1}{1+zq^N}\sum_{m\geq 1}\frac{(-1)^mq^{m(2N+1)}z^m}{1-q^{2m-1}}.
\end{align*}
For the right-hand side, set
\begin{align*}
A(x)=\sum_{k=0}^\infty \frac{x^kq^{k}}{1+zq^{2k+1}},~B(x)=\sum_{k=0}^\infty \frac{x^kq^{2k}}{1+zq^{2k+1}},~C=\sum_{k=0}^\infty \frac{q^{k}}{1+zq^{2k+1}},~D(x)=\sum_{k=0}^\infty \frac{x^kq^{k}}{1+zq^{k+1}}.
\end{align*}
Then the right-hand side equals $xzq^2(A(x)B(x)-CD(x))$; its $x^N$ coefficient is
\begin{align*}
[x^N]\big(xzq^2(AB-CD)\big)=zq^2\big([x^{N-1}](AB)-C[x^{N-1}]D(x)\big).
\end{align*}
Since $[x^{N-1}]D(x)=\dfrac{q^{N-1}}{1+zq^N}$ and
\begin{align*}
[x^{N-1}](AB)=\sum_{j=0}^{N-1}\left(\frac{q^j}{1+zq^{2j+1}}\cdot\frac{q^{2(N-1-j)}}{1+zq^{2N-1-2j}}\right),
\end{align*}
we obtain
\begin{align*}
[x^N]\text{RHS}=\frac{zq^2}{1+zq^N}\Big((1+zq^N)\sum_{j=0}^{N-1}\frac{q^{2N-2-j}}{(1+zq^{2j+1})(1+zq^{2N-1-2j})}-Cq^{N-1}\Big).
\end{align*}
Thus the desired identity is equivalent, after multiplying by $(1+zq^N)$, to
\begin{align}\label{eq-id}
\sum_{m\geq 1}\frac{(-1)^mq^{m(2N+1)}z^m}{1-q^{2m-1}}=zq^2(1+zq^N)\sum_{j=0}^{N-1}\left(\frac{q^j}{1+zq^{2j+1}}\cdot\frac{q^{2(N-1-j)}}{1+zq^{2N-1-2j}}\right)-zq^{N+1}C.
\end{align}

We denote the finite sum on the right-hand side of \eqref{eq-id} as $F(q)$. Then we have
\begin{align*}
F(q)=&\sum_{j=0}^{N-1}\frac{q^{2N-j}}{q^{2N-1-2j}-q^{2j+1}}\left(\frac{1+zq^N}{1+zq^{2j+1}}-\frac{1+zq^N}{1+zq^{2N-1-2j}}\right)\\
=&\sum_{j=0}^{N-1}\frac{zq^{2N-j}}{q^{2N-1-2j}-q^{2j+1}}\left(\frac{q^N-q^{2j+1}}{1+zq^{2j+1}}-\frac{q^N-q^{2N-1-2j}}{1+zq^{2N-1-2j}}\right).
\end{align*}
Re-indexing the second part $(j\mapsto N-1-j)$ makes it identical to the first, and we obtain
\begin{align*}
F(q)=&\sum_{j=0}^{N-1}\frac{z(q^{2N-j}+q^{N+1+j})(q^N-q^{2j+1})}{(q^{2N-1-2j}-q^{2j+1})(1+zq^{2j+1})}=\sum_{j=0}^{N-1}\frac{zq^{N+3j+2}(q^{N-2j-1}+1)(q^{N-2j-1}-1)}{q^{2j+1}(q^{2N-4j-2}-1)(1+zq^{2j+1})}\\
=&\sum_{j=0}^{N-1}\frac{zq^{N+j+1}}{1+zq^{2j+1}}=zq^{N+1}C-\sum_{j=N}^{\infty}\frac{zq^{N+j+1}}{1+zq^{2j+1}}.
\end{align*}
Thus the right-hand side of \eqref{eq-id} is equal to
\begin{align*}
&-\sum_{j=N}^{\infty}\frac{zq^{N+j+1}}{1+zq^{2j+1}}=-\sum_{j=0}^{\infty}\frac{zq^{2N+j+1}}{1+zq^{2j+2N+1}}=-\sum_{k,j\geq0}zq^{2N+1+j}(-zq^{2j+2N+1})^k\\
=&\sum_{k\geq0}(-z)^{k+1}q^{2Nk+2N+k+1}\sum_{j\geq0}q^{(2k+1)j}=\sum_{k\geq0}\frac{(-z)^{k+1}q^{(2N+1)(k+1)}}{1-q^{2k+1}}=\sum_{k\geq 1}\frac{(-1)^kq^{k(2N+1)}z^k}{1-q^{2k-1}}.
\end{align*}
Therefore, we complete the proof of \eqref{eq-id}.
\end{proof}

Setting $x=z=1$ and simplifying, we obtain the following corollary.
\begin{cor}
\begin{align*}
Y(q)=-q\frac{(q^4;q^4)_\infty^4}{(q^2;q^2)_\infty^2}\sum_{k=1}^\infty \frac{q^{2k}}{1+q^{2k}},
\end{align*}
and therefore $Y(q)$ is an odd function of $q$.
\end{cor}

%As shown in [??] and [??], Example \ref{r:adsy} can also be derived from the repeated application of Properties \ref{r:prop1} and \ref{r:prop2}, which suggests a way to approach the identity. However, to avoid cumbersome calculations and discussions on convergence, we prove the identity by comparing coefficients.

\subsection{The Amdeberhan--Andrews--Ballantine Conjecture I}
We present a generalized form of Conjecture \ref{AAB-conj-I} through the following example.
\begin{example}
For $|y|<1$, $|q|<|x|<1/|y|$, $z\neq q^{-r}$, $w\neq q^{-r}$, $w\neq0$, and $z\neq wq^{-r-1}$ for $r\in\mathbb{N}$, we have
\begin{align}\label{can-conj1}
&A(x,y,z,w)-\frac{zq}{xw}A\left(xy,\frac{q}{x},\frac{zq}{w},z\right)\\\nonumber
=&L(xy,z,w)+yL\left(y,\frac{zq}{w}\right)L(xy,w)-\frac{zq}{xw}L\left(\frac{q}{x},\frac{zq}{w}\right)L(xy,z).
\end{align}
\end{example}

\begin{proof}
Fix $M,~N\in \mathbb{Z}$ and compare the coefficients of $x^Ny^M$ on both sides. For the left-hand side,
\begin{align*}
[x^Ny^M]A(x,y,z,w)=
\begin{cases}
\dfrac{1}{(1-wq^N)(1-zq^{M})}&M\geq N\geq 0\\
0&\text{otherwise}
\end{cases}
\end{align*}
and
\begin{align*}
[x^Ny^M]\frac{zq}{wx}A\left(xy,\frac{q}{x},\frac{zq}{w},z\right)=
\begin{cases}
\dfrac{zw^{-1}q^{M-N}}{(1-zw^{-1}q^{M-N})(1-zq^{M})}&M\geq0~\text{and}~N\leq -1\\
0&\text{otherwise}.
\end{cases}
\end{align*}
For the right-hand side,
\begin{align*}
[x^Ny^M]L(xy,z,w)=
\begin{cases}
\dfrac{1}{(1-wq^N)(1-zq^{N})}&M=N\geq 0\\
0&\text{otherwise},
\end{cases}
\end{align*}
\begin{align*}
[x^Ny^M]yL\left(y,\frac{zq}{w}\right)L(xy,w)=
\begin{cases}
\dfrac{1}{(1-wq^N)(1-zw^{-1}q^{M-N})}&M\geq N+1~\text{and}~N\geq 0\\
0&\text{otherwise},
\end{cases}
\end{align*}
and
\begin{align*}
[x^Ny^M]\frac{zq}{xw}L\left(\frac{q}{x},\frac{zq}{w}\right)L(xy,z)=
\begin{cases}
\dfrac{zw^{-1}q^{M-N}}{(1-zq^M)(1-zw^{-1}q^{M-N})}&M\geq N+1~\text{and}~M\geq 0\\
0&\text{otherwise}.
\end{cases}
\end{align*}

We can easily observe that the coefficients of $x^Ny^M$ on both sides are not equal to 0 only when $M\geq0$ and $M\geq N$. Therefore, we will discuss the coefficients in the following several cases.

Case 1 $(N\geq0)$: When $M=N\geq0$, we obtain
\begin{align*}
[x^Ny^M]\text{LHS}=\frac{1}{(1-wq^N)(1-zq^{N})}=[x^Ny^M]\text{RHS}.
\end{align*}
When $M\geq N+1$ and $M,N\geq0$, we obtain
\begin{align*}
[x^Ny^M]\text{RHS}=&\frac{1}{(1-wq^N)(1-zw^{-1}q^{M-N})}-\frac{zw^{-1}q^{M-N}}{(1-zq^M)(1-zw^{-1}q^{M-N})}\\
=&\frac{(1-zq^M)-zw^{-1}q^{M-N}(1-wq^N)}{(1-wq^N)(1-zq^M)(1-zw^{-1}q^{M-N})}\\
=&\frac{1-zq^M-zw^{-1}q^{M-N}+zq^M}{(1-wq^N)(1-zq^M)(1-zw^{-1}q^{M-N})}\\
=&\frac{1}{(1-wq^N)(1-zq^{N})}=[x^Ny^M]\text{LHS}.
\end{align*}

Case 2 $(N<0)$: Since the global condition requires $M\geq0$ and our current case postulates $N<0$, it strictly follows that the relation $M>N$ is inherently satisfied. Consequently, we obtain
\begin{align*}
[x^Ny^M]\text{LHS}=-\frac{zw^{-1}q^{M-N}}{(1-zq^M)(1-zw^{-1}q^{M-N})}=[x^Ny^M]\text{RHS}.
\end{align*}
Therefore, the coefficients of $x^Ny^M$ on both sides of the identity for all $M,N\in\mathbb{Z}$ are equal.
\end{proof}

Setting $xy=q$ and replacing both $z$ and $w$ with $zq$, we obtain the following corollary as a generalized form of Conjecture \ref{AAB-conj-I}.
\begin{cor}\label{g:AAB-conj-I}
For $|x|>|q|$, $|zq|<1$ and each positive integer $n$, we have
\begin{align*}
[q^n]\sum_{n\geq0}\sum_{m\geq n}\left(\frac{zx^{n-m}q^{m+1}}{(1-zq^{m+1})(1-zq^{n+1})}-\frac{zx^{-m-1}q^{n+m+2}}{(1-q^{m+1})(1-zq^{n+1})}\right)=\sum_{d|n}dz^d.
\end{align*}
\end{cor}
\begin{proof}
We can observe that the series on the left-hand side is
\begin{align*}
K(q):=zqA\left(x,\frac{q}{x},zq,zq\right)-\frac{zq^2}{x}A\left(q,\frac{q}{x},q,zq\right).
\end{align*}
From \eqref{can-conj1}, we have
\begin{align*}
K(q)=&zqL(q,zq,zq)=\sum_{n\geq0}\frac{zq^{n+1}}{(1-zq^{n+1})^2}=\sum_{n,k\geq0}(k+1)z^{k+1}q^{nk+n+k+1}\\
=&\sum_{n,k\geq1}kz^kq^{nk}=\sum_{n=1}^\infty\Big(\sum_{d|n}dz^d\Big)q^n.
\end{align*}
\end{proof}

We note that Corollary \ref{g:AAB-conj-I} can yield Conjecture \ref{AAB-conj-I} immediately. It would be more desirable, however, to have the left-hand side of Corollary \ref{g:AAB-conj-I} written in a form which is the coefficients of $q^{n2^a}$, as in Conjecture \ref{AAB-conj-I}. Next, we show how to obtain Conjecture \ref{AAB-conj-I} by Corollary \ref{g:AAB-conj-I}.

First, rearranging and using properties of geometric series, we obtain
\begin{align*}
&\sum_{m,n\geq 1}\frac{q^{mn2^a}}{(1+q^{n2^{a-1}})(1-q^{2m-1})}=\sum_{m,n\geq 1}\frac{1}{1+q^{n2^{a-1}}}\sum_{k\geq 0}q^{mn2^a+(2m-1)k}\\
=&\sum_{r=0}^{2^{a-1}-1}\sum_{m,n\geq 1}\frac{1}{1+q^{n2^{a-1}}}\sum_{k\geq 0}q^{mn2^a+(2m-1)(2^{a-1}k+r)}\\
=&\sum_{r=0}^{2^{a-1}-1}\sum_{m,n\geq 1}\frac{1}{1+q^{n2^{a-1}}}\sum_{k\geq 0}q^{mn2^a+mk2^a-k2^{a-1}+(2m-1)r}
\end{align*}
We can observe that the power of $q$ in each term of the above series is congruent to $(2m-1)r$ modulo $2^{a-1}$. Therefore, we obtain
\begin{align*}
&[q^{n2^a}]\sum_{m,n\geq 1}\frac{q^{mn2^a}}{(1+q^{n2^{a-1}})(1-q^{2m-1})}=[q^{n2^a}]\sum_{m,n\geq 1}\frac{1}{1+q^{n2^{a-1}}}\sum_{k\geq 0}q^{mn2^a+mk2^a-k2^{a-1}}\\
=&[q^{n2^a}]\sum_{n\geq 1}\sum_{k\geq 0}\left(\frac{q^{-k2^{a-1}}}{1+q^{n2^{a-1}}}\cdot\frac{q^{n2^a+k2^a}}{1-q^{n2^a+k2^a}}\right)=[q^{2n}]\sum_{n,k\geq0}\frac{q^{2n+k+2}}{(1+q^{n+1})(1-q^{2n+2k+2})}\\
=&[q^{2n}]\left(\sum_{n,k\geq0}\frac{q^{2n+k+2}}{(1-q^{2n+2})(1-q^{2n+2k+2})}-\sum_{n,k\geq0}\frac{q^{3n+k+3}}{(1-q^{2n+2})(1-q^{2n+2k+2})}\right)\\
=&[q^{2n}]\left(\sum_{n\geq0}\sum_{k\geq n}\frac{q^{n+k+2}}{(1-q^{2n+2})(1-q^{2k+2})}-\sum_{n\geq0}\sum_{k\geq n}\frac{q^{2n+k+3}}{(1-q^{2n+2})(1-q^{2k+2})}\right).
\end{align*}
Thus Conjecture \ref{AAB-conj-I} is equivalent to
\begin{align}\label{A-A=L}
q^2A(q,q,q^2,q^2;q^2)-q^3A(q^2,q,q^2,q^2;q^2)=\sum_{n=1}^\infty\sigma_1(n)q^{2n}.
\end{align}
Then replacing $q$ with $q^2$ and setting $x=q$, $z=1$ in Corollary \ref{g:AAB-conj-I}, we can obtain \eqref{A-A=L}.

\subsection{The Amdeberhan--Andrews--Ballantine Conjecture II}
We find the following formula to handle the double Lambert series, such as Conjecture \ref{AAB-conj-II}.
\begin{example}
For $|q|<|x|,|y|,|z|,|w|<1$, $|xy|>|q|$, and $w\neq zq^{-r}$ for $r\in\mathbb{N}$, we have
\begin{align}\label{r:prop3}
&A(x,y,zq,w)-\frac{q^2}{xy^2wz}A\left(\frac{q}{x},\frac{q}{y},\frac{q}{z},\frac{q}{w}\right)\\\nonumber
=&\frac{1}{y}L(x,w)L^*(y,z)-\frac{1}{y}L\left(x,\frac{w}{z}\right)L^*(xy,z)+\frac{w}{yz}L\left(\frac{q}{y},\frac{w}{z}\right)L^*(xy,w).
\end{align}
\end{example}
\begin{proof}
Fix $M,N\in\mathbb{Z}$ and compare the coefficients of $x^Ny^M$ on both sides. For the left-hand side,
\begin{align*}
[x^Ny^M]A(x,y,zq,w)=
\begin{cases}
\dfrac{1}{(1-wq^N)(1-zq^{M+1})}&M\geq N\geq 0\\
0&\text{otherwise}
\end{cases}
\end{align*}
and
\begin{align*}
[x^Ny^M]\frac{q^2}{xy^2wz}A\left(\frac{q}{x},\frac{q}{y},\frac{q}{z},\frac{q}{w}\right)=
\begin{cases}
\dfrac{q^{-M-N-1}}{wz(1-w^{-1}q^{-N})(1-z^{-1}q^{-M-1})}&M+1\leq N\leq -1\\
0&\text{otherwise}.
\end{cases}
\end{align*}
For the right-hand side,
\begin{align*}
[x^Ny^M]\frac{1}{y}L(x,w)L^*(y,z)=
\begin{cases}
\dfrac{1}{(1-wq^N)(1-zq^{M+1})}&N\geq 0\\
0&\text{otherwise},
\end{cases}
\end{align*}
\begin{align*}
[x^Ny^M]\frac{1}{y}L\left(x,\frac{w}{z}\right)L^*(xy,z)=[x^N]\frac{x^{M+1}}{1-zq^{M+1}}\sum_{n=0}^\infty\frac{x^n}{1-wz^{-1}q^n}\\
=\begin{cases}
\dfrac{1}{(1-wz^{-1}q^{N-M-1})(1-zq^{M+1})}&N\geq M+1\\
0&\text{otherwise},
\end{cases}
\end{align*}
and
\begin{align*}
[x^Ny^M]\frac{w}{yz}L\left(\frac{q}{y},\frac{w}{z}\right)L^*(xy,w)=[y^{M+1}]\frac{wz^{-1}y^N}{1-wq^N}\sum_{n=0}^\infty\frac{q^ny^{-n}}{1-wz^{-1}q^n}\\
=\begin{cases}
\dfrac{wq^{N-M-1}}{z(1-wz^{-1}q^{N-M-1})(1-wq^N)}&N\geq M+1\\
0&\text{otherwise}.
\end{cases}
\end{align*}

Case 1 $(M\geq N)$: When $M\geq N\geq 0$, we obtain
\begin{align*}
[x^Ny^M]\text{LHS}=[x^Ny^M]\text{RHS}=\frac{1}{(1-wq^N)(1-zq^{M+1})}.
\end{align*}
When $M\geq N$ and $N<0$, we obtain $[x^Ny^M]\text{LHS}=[x^Ny^M]\text{RHS}=0$.

Case 2 $(M+1\leq N)$: When $M+1\leq N\leq -1$, we obtain
\begin{align*}
[x^Ny^M]\text{LHS}=-\frac{q^{-M-N-1}}{wz(1-w^{-1}q^{-N})(1-z^{-1}q^{-M-1})}=-\frac{1}{(1-wq^N)(1-zq^{M+1})}
\end{align*}
and
\begin{align*}
[x^Ny^M]\text{RHS}=&-\frac{1}{(1-wz^{-1}q^{N-M-1})(1-zq^{M+1})}+\frac{wq^{N-M-1}}{z(1-wz^{-1}q^{N-M-1})(1-wq^N)}\\
=&\frac{-(1-wq^N)+wz^{-1}q^{N-M-1}(1-zq^{M+1})}{(1-wz^{-1}q^{N-M-1})(1-wq^N)(1-zq^{M+1})}\\
=&-\frac{1}{(1-wq^N)(1-zq^{M+1})}=[x^Ny^M]\text{LHS}.
\end{align*}
When $M+1\leq N$ and $N\geq 0$, we obtain
\begin{align*}
&[x^Ny^M]\text{RHS}\\
=&\dfrac{1}{(1-wq^N)(1-zq^{M+1})}-\dfrac{1}{(1-wz^{-1}q^{N-M-1})(1-zq^{M+1})}+\dfrac{wq^{N-M-1}}{z(1-wz^{-1}q^{N-M-1})(1-wq^N)}\\
=&\frac{1}{(1-wq^N)(1-zq^{M+1})}-\frac{1}{(1-wq^N)(1-zq^{M+1})}=0=[x^Ny^M]\text{LHS}.
\end{align*}
Therefore, the coefficients of $x^Ny^M$ on both sides of the identity for all $M,N\in\mathbb{Z}$ are equal.
\end{proof}

Next, we will demonstrate how to use the above example to obtain Conjecture \ref{AAB-conj-II}. We define the following auxiliary function
\begin{align*}
X(q):=\sum_{n\geq2}\frac{(n-1)q^{n-2}}{1+q^{2n-1}}.
\end{align*}
Then we can easily observe that the series on the left-hand side of the identity is an even function and the series on the right-hand side of the identity is $q^2X(q)$. Rearranging and using properties of geometric series, we obtain
\begin{align*}
&\sum_{m,n\geq1}\frac{q^{2mn}}{(1+q^{2n-1})(1-q^{2m-1})}=\sum_{m,n\geq1}\frac{q^{2mn}}{1+q^{2n-1}}\sum_{i\geq0}q^{2mi-i}\\
=&\sum_{n\geq1}\sum_{i\geq0}\frac{q^{-i}}{1+q^{2n-1}}\sum_{m\geq1}q^{2mn+2mi}=\sum_{n\geq1}\sum_{i\geq0}\left(\frac{q^{-i}}{1+q^{2n-1}}\cdot\frac{q^{2n+2i}}{1-q^{2n+2i}}\right)\\
=&\sum_{n,i\geq0}\frac{q^{2n+i+2}}{(1+q^{2n+1})(1-q^{2n+2i+2})}=\sum_{n\geq0}\sum_{i\geq n}\frac{q^{n+i+2}}{(1+q^{2n+1})(1-q^{2i+2})}\\
=&q^2A(q,q,q^2,-q;q^2).
\end{align*}
Thus Conjecture \ref{AAB-conj-II} is equivalent to
\begin{align}\label{2A=L+L}
2A(q,q,q^2,-q;q^2)=X(q)+X(-q).
\end{align}

First, replacing $q$ with $q^2$ and setting $w=-q$, $y=q$ in \eqref{r:prop3}, we obtain
\begin{align}\label{prop3-step1}
&A(x,q,zq^2,-q;q^2)+\frac{q}{xz}A\left(\frac{q^2}{x},q,\frac{q^2}{z},-q;q^2\right)\\\nonumber
=&\frac{1}{q}L(x,-q;q^2)L^*(q,z;q^2)-\frac{1}{q}L\left(x,\frac{-q}{z};q^2\right)L^*(xq,z;q^2)-\frac{1}{z}L\left(q,\frac{-q}{z};q^2\right)L^*(xq,-q;q^2).
\end{align}
For the last two items on the right-hand side of \eqref{prop3-step1}, we have
\begin{align*}
&\frac{1}{q}L\left(x,\frac{-q}{z};q^2\right)L^*(xq,z;q^2)+\frac{1}{z}L\left(q,\frac{-q}{z};q^2\right)L^*(xq,-q;q^2)\\
=&\frac{1}{q}L\left(x,\frac{-q}{z};q^2\right)L^*(xq,z;q^2)-\frac{1}{q}L\left(q,\frac{-q}{z};q^2\right)L^*(xq,z;q^2)\\
&\qquad+\frac{1}{q}L\left(q,\frac{-q}{z};q^2\right)L^*(xq,z;q^2)+\frac{1}{z}L\left(q,\frac{-q}{z};q^2\right)L^*(xq,-q;q^2)
\end{align*}
We note that
\begin{align*}
&\lim_{x\rightarrow q}\frac{1}{q}L^*(xq,z;q^2)\left[L\left(x,\frac{-q}{z};q^2\right)-L\left(q,\frac{-q}{z};q^2\right)\right]\\
=&\lim_{x\rightarrow q}\frac{x-q}{q}L^*(xq,z;q^2)\cdot\lim_{x\rightarrow q}\dfrac{L(x,-q/z;q^2)-L(q,-q/z;q^2)}{x-q}\\
=&\lim_{x\rightarrow q}\frac{x-q}{q}\frac{(q^2,q^2,xzq,q/xz;q^2)_\infty}{(xq,q/x,z,q^2/z;q^2)_\infty}\cdot\left.\frac{d}{dx}L(x,-q/z;q^2)\right|_{x=q}\\
=&\lim_{x\rightarrow q}\frac{x-q}{q}\frac{(q^2,q^2,xzq,q/xz;q^2)_\infty}{(1-q/x)\cdot(xq,q^3/x,z,q^2/z;q^2)_\infty}\cdot\sum_{n=1}^\infty\frac{nq^{n-1}}{1+q^{2n+1}/z}\\
=&-\frac{1}{z}\sum_{n=1}^\infty\frac{nq^{n-1}}{1+q^{2n+1}/z}
\end{align*}
Moreover, we have
\begin{align*}
&\frac{1}{q}L^*(xq,z;q^2)+\frac{1}{z}L^*(xq,-q;q^2)\\
=&\frac{1}{q}\sum_{n=0}^\infty\frac{x^nq^n}{1-zq^{2n}}+\frac{1}{q}\sum_{n=1}^\infty\frac{x^{-n}q^n}{q^{2n}-z}+\frac{1}{z}\sum_{n=0}^\infty\frac{x^nq^n}{1+q^{2n+1}}+\frac{1}{z}\sum_{n=1}^\infty\frac{x^{-n}q^{n-1}}{1+q^{2n-1}}
\end{align*}
For the second and fourth items, we first find the common denominator and then set $x\rightarrow q$ to obtain
\begin{align*}
&\lim_{x\rightarrow q}\left[\frac{1}{q}\sum_{n=1}^\infty\frac{x^{-n}q^n}{q^{2n}-z}+\frac{1}{z}\sum_{n=1}^\infty\frac{x^{-n}q^{n-1}}{1+q^{2n-1}}\right]=\lim_{x\rightarrow q}\frac{1}{q}\sum_{n=1}^\infty\frac{x^{-n}q^n(zq^{2n-1}+q^{2n})}{z(q^{2n}-z)(1+q^{2n-1})}\\
=&\lim_{x\rightarrow q}\frac{z+q}{zq}\sum_{n=1}^\infty\frac{x^{-n}q^nq^{2n-1}}{(q^{2n}-z)(1+q^{2n-1})}=\frac{z+q}{zq}\sum_{n=1}^\infty\frac{q^{2n-1}}{(q^{2n}-z)(1+q^{2n-1})}
\end{align*}
Then setting $x\rightarrow q$ in \eqref{prop3-step1}, we have
\begin{align}\label{prop3-step2}
&A(q,q,zq^2,-q;q^2)+\frac{1}{z}A\left(q,q,\frac{q^2}{z},-q;q^2\right)=\frac{1}{q}L(q,-q;q^2)L^*(q,z;q^2)+\frac{1}{z}\sum_{n=1}^\infty\frac{nq^{n-1}}{1+q^{2n+1}/z}\\\nonumber
&\qquad-L\left(q,\frac{-q}{z};q^2\right)\left[\frac{1}{q}\sum_{n=0}^\infty\frac{q^{2n}}{1-zq^{2n}}+\frac{1}{z}\sum_{n=0}^\infty\frac{q^{2n}}{1+q^{2n+1}}+\frac{z+q}{zq}\sum_{n=1}^\infty\frac{q^{2n-1}}{(q^{2n}-z)(1+q^{2n-1})}\right]
\end{align}
We focus on the first and third items on the right-hand side of \eqref{prop3-step2}
\begin{align*}
&\lim_{z\rightarrow 1}\left[L(q,-q;q^2)L^*(q,z;q^2)-L\left(q,\frac{-q}{z};q^2\right)\sum_{n=0}^\infty\frac{q^{2n}}{1-zq^{2n}}\right]\\
=&\lim_{z\rightarrow 1}\frac{L(q,-q;q^2)-L(q,-q/z;q^2)}{1-z}+L(q,-q;q^2)\left(\sum_{n\in\mathbb{Z},n\neq0}\frac{q^n}{1-q^{2n}}-\sum_{n\geq1}\frac{q^{2n}}{1-q^{2n}}\right)\\
=&\left.\frac{d}{dz}L(q,-q/z;q^2)\right|_{z=1}-L(q,-q;q^2)\sum_{n\geq1}\frac{q^{2n}}{1-q^{2n}}\\
=&\sum_{n=0}^\infty\frac{q^{3n+1}}{(1+q^{2n+1})^2}-L(q,-q;q^2)\sum_{n\geq1}\frac{q^{2n}}{1-q^{2n}}.
\end{align*}
Thus setting $z\rightarrow 1$ in \eqref{prop3-step2}, we have
\begin{align*}
&2A(q,q,q^2,-q;q^2)=\sum_{n=1}^\infty\frac{nq^{n-1}}{1+q^{2n+1}}+\sum_{n=0}^\infty\frac{q^{3n}}{(1+q^{2n+1})^2}\\
&\qquad+L(q,-q;q^2)\left[-\sum_{n\geq1}\frac{q^{2n-1}}{1-q^{2n}}-\sum_{n\geq0}\frac{q^{2n}}{1+q^{2n+1}}+\frac{1+q}{q}\sum_{n=1}^\infty\frac{q^{2n-1}}{(1-q^{2n})(1+q^{2n-1})}\right].
\end{align*}
For the auxiliary function $X(q)$, we obtain
\begin{align*}
X(q)=&\sum_{n\geq1}\frac{nq^{n-1}}{1+q^{2n+1}}=\sum_{n\geq1}\sum_{m\geq0}(-1)^mnq^{2mn+m+n-1}\\
=&\sum_{m\geq0}\frac{(-1)^mq^{m-1+2m+1}}{(1-q^{2m+1})^2}=\sum_{n=0}^\infty\frac{(-1)^nq^{3n}}{(1-q^{2n+1})^2}.
\end{align*}
By finding the common denominator, we obtain
\begin{align*}
-\sum_{n\geq1}\frac{q^{2n-1}}{1-q^{2n}}-\sum_{n\geq0}\frac{q^{2n}}{1+q^{2n+1}}+\frac{1+q}{q}\sum_{n=1}^\infty\frac{q^{2n-1}}{(1-q^{2n})(1+q^{2n-1})}=0.
\end{align*}
Then we can complete the proof of \eqref{2A=L+L}.

\subsection{A result of Amdeberhan, Andrews, and Ballantine}
We present the following lemma to demonstrate a new proof of a result of Amdeberhan, Andrews, and Ballantine \cite[Theorem 3.4]{amanba26}.

\begin{lemma}
For $|x|<1$, we have
\begin{align}\label{f3=A'}
xA(x,q,q^2,q)=&\frac{x^2}{2q^2}[L(q^2,x,x)-L(x,q)^2]+\frac{x}{q}[L(x,q)L(q,q)-L(xq,q,q)].
\end{align}
\end{lemma}

\begin{proof}
Fix $N\geq1$ and compare the coefficients of $x^N$ on both sides. For the left-hand side,
\begin{align*}
[x^N]xA(x,q,q^2,q)=[x^N]\sum_{n=0}^\infty\sum_{m\geq n}\frac{x^{n+1}q^m}{(1-q^{n+1})(1-q^{m+2})}=\frac{1}{1-q^N}\sum_{m\geq N-1}\frac{q^m}{1-q^{m+2}}.
\end{align*}
For the right-hand side, since
\begin{align}\label{Lq^2xx}
L(q^2,x,x)=\sum_{n\geq0}\frac{q^{2n}}{(1-xq^n)^2}=\sum_{k,n\geq0}(k+1)x^kq^{nk+2n}=\sum_{k\geq0}\frac{(k+1)x^k}{1-q^{k+2}},
\end{align}
we obtain
\begin{align*}
[x^N]\frac{x^2}{2q^2}L(q^2,x,x)=
\begin{cases}
\dfrac{N-1}{2q^2(1-q^N)}&N\geq 2\\
0&N=1.
\end{cases}
\end{align*}
For each remaining terms on the right-hand side, we have
\begin{align*}
[x^N]\frac{-x^2}{2q^2}L(x,q)^2=
\begin{cases}
-\displaystyle\sum_{i=0}^{N-2}\dfrac{1}{2q^2(1-q^{i+1})(1-q^{N-i-1})}&N\geq 2\\
0&N=1,
\end{cases}
\end{align*}
\begin{align*}
[x^N]\frac{x}{q}L(x,q)L(q,q)=\frac{1}{q(1-q^N)}L(q,q)~~\text{for}~~N\geq1,
\end{align*}
and
\begin{align*}
[x^N]-\frac{x}{q}L(xq,q,q)=-\frac{q^{N-1}}{q(1-q^N)^2}~~\text{for}~~N\geq1.
\end{align*}
When $N=1$, we obtain
\begin{align*}
&[x^N]\text{LHS}=\frac{1}{1-q}\sum_{m\geq0}\frac{q^m}{1-q^{m+2}}=\frac{1}{q(1-q)}\sum_{m\geq1}\frac{q^m}{1-q^{m+1}}\\
=&\frac{1}{q(1-q)}L(q,q)-\frac{1}{q(1-q)^2}=[x^N]\text{RHS}.
\end{align*}
When $N\geq 2$, the desired identity is equivalent, after multiplying by $(1-q^N)$, to
\begin{align}\label{eq-id2}
\sum_{m\geq N-1}\frac{q^m}{1-q^{m+2}}=\frac{N-1}{2q^2}-\sum_{i=1}^{N-1}\frac{1-q^N}{2q^2(1-q^{i})(1-q^{N-i})}+\frac{1}{q}L(q,q)-\frac{q^{N-2}}{1-q^N}.
\end{align}

We denote the finite sum on the right-hand side of \eqref{eq-id2} as $G(q)$. Then we have
\begin{align*}
&\sum_{i=1}^{N-1}\frac{1-q^N}{(1-q^{i})(1-q^{N-i})}=\sum_{i=1}^{N-1}\left(\frac{1}{1-q^{N-i}}+\frac{q^i}{1-q^i}\right)\\
=&\sum_{i=1}^{N-1}\left(\frac{1}{1-q^{i}}+\frac{q^i}{1-q^i}\right)=N-1+\sum_{i=1}^{N-1}\frac{2q^i}{1-q^i}.
\end{align*}
Here we re-index the first part $(j\mapsto N-j)$ for the penultimate equality. Then the right-hand side of \eqref{eq-id2} is equal to
\begin{align*}
&-\sum_{i=1}^{N-1}\frac{q^{i-2}}{1-q^i}+\frac{1}{q}L(q,q)-\frac{q^{N-2}}{1-q^N}=\sum_{i\geq0}\frac{q^{i-1}}{1-q^{i+1}}-\sum_{i=1}^N\frac{q^{i-2}}{1-q^i}\\
=&\sum_{i\geq0}\frac{q^{i-1}}{1-q^{i+1}}-\sum_{i=0}^{N-1}\frac{q^{i-1}}{1-q^{i+1}}=\sum_{i\geq N}\frac{q^{i-1}}{1-q^{i+1}}=\sum_{i\geq N-1}\frac{q^i}{1-q^{i+2}}.
\end{align*}
Thus we complete the proof of \eqref{eq-id2}.
\end{proof}

Amdeberhan, Andrews, and Ballantine \cite{amanba26} proved the following identity connecting the general and double Lambert series.

\begin{example}\cite[Theorem 3.4]{amanba26}\label{f1=f3}
We denote the following two series:
\begin{align*}
f_1(q):=\sum_{k\geq1}\left(\frac{k(k-1)q^k}{1-q^k}-\frac{2kq^{2k}}{(1-q^k)^2}\right)
\end{align*}
and
\begin{align*}
f_3(q):=\sum_{k\geq1}\sum_{l>k}\frac{2kq^{k+l}}{(1-q^k)(1-q^l)}.
\end{align*}
Then we have $f_1(q)=f_3(q)$.
\end{example}

We observe that $f_1(q)$ is related to the divisor function $\sigma_2(n)$ and $f_3(q)$ can be obtained from the derivative of the double Lambert series $A(x,y,z,w)$. Thus we obtain the following proof.
\begin{proof}
We take the derivative of both sides of \eqref{f3=A'} with respect to the variable $x$. For the left-hand side, we obtain
\begin{align*}
\frac{d}{dx}xA(x,q,q^2,q)=\frac{d}{dx}\sum_{n\geq1}\sum_{m>n}\frac{x^nq^{m-2}}{(1-q^n)(1-q^m)}=\sum_{n\geq1}\sum_{m>n}\frac{nx^{n-1}q^{m-2}}{(1-q^n)(1-q^m)}.
\end{align*}
For the right-hand side, from \eqref{Lq^2xx}, we obtain
\begin{align*}
\frac{d}{dx}\frac{x^2}{2q^2}L(q^2,x,x)=\frac{d}{dx}\frac{1}{2q^2}\sum_{k\geq0}\frac{(k+1)x^{k+2}}{1-q^{k+2}}=\frac{1}{2q^2}\sum_{k\geq0}\frac{(k+2)(k+1)x^{k+1}}{1-q^{k+2}}.
\end{align*}
For each remaining terms on the right-hand side, we have
\begin{align*}
\frac{d}{dx}\frac{x^2}{2q^2}L(x,q)^2=\frac{d}{dx}\frac{1}{2q^2}\left(\sum_{n\geq0}\frac{x^{n+1}}{1-q^{n+1}}\right)^2=\frac{1}{q^2}\sum_{n\geq0}\frac{x^{n+1}}{1-q^{n+1}}\cdot\sum_{m\geq0}\frac{(m+1)x^{m}}{1-q^{m+1}},
\end{align*}
\begin{align*}
\frac{d}{dx}\frac{x}{q}L(x,q)L(q,q)=\frac{d}{dx}\left(\sum_{m\geq0}\frac{x^{m+1}}{1-q^{m+1}}\cdot\sum_{n\geq0}\frac{q^{n-1}}{1-q^{n+1}}\right)=\sum_{m\geq0}\frac{(m+1)x^{m}}{1-q^{m+1}}\cdot\sum_{n\geq0}\frac{q^{n-1}}{1-q^{n+1}},
\end{align*}
and
\begin{align*}
\frac{d}{dx}\frac{x}{q}L(xq,q,q)=\frac{d}{dx}\sum_{n\geq0}\frac{x^{n+1}q^{n-1}}{(1-q^{n+1})^2}=\sum_{n\geq0}\frac{(n+1)x^{n}q^{n-1}}{(1-q^{n+1})^2}.
\end{align*}
Then setting $x=q$ and multiplying by $2q^3$, we obtain
\begin{align*}
\text{LHS}=\sum_{n\geq1}\sum_{m>n}\frac{2nq^{m+n}}{(1-q^n)(1-q^m)}=f_3(q)
\end{align*}
and
\begin{align*}
\text{RHS}=&\sum_{k\geq0}\frac{(k+2)(k+1)q^{k+2}}{1-q^{k+2}}-\sum_{n\geq0}\frac{2q^{n+2}}{1-q^{n+1}}\cdot\sum_{m\geq0}\frac{(m+1)q^{m}}{1-q^{m+1}}\\
&\qquad\qquad+\sum_{m\geq0}\frac{(m+1)q^{m}}{1-q^{m+1}}\cdot\sum_{n\geq0}\frac{2q^{n+2}}{1-q^{n+1}}-\sum_{n\geq0}\frac{2(n+1)q^{2n+2}}{(1-q^{n+1})^2}\\
=&\sum_{k\geq2}\frac{k(k-1)q^{k}}{1-q^{k}}-\sum_{n\geq1}\frac{2nq^{2n}}{(1-q^{n})^2}=f_1(q).
\end{align*}
Then we obtain $f_1(q)=f_3(q)$.
\end{proof}

At the conclusion of \cite{ks26}, Kumar and Singh asked whether the divisor function $\sigma_k(n)$ can arise as the coefficients of a double Lambert series similar to the one for $\sigma_1(n)$ in Conjecture \ref{AAB-conj-I}. We believe that this method of derivation is a correct approach to answering this question.

\subsection*{Acknowledgements}
The first author was supported by the National Natural Science Foundation of China (Grant No. 12401438).

\subsection*{Data availability}
No data was used for the research described in the article.

%%%%%%%%%%%%%%%%%%%%%%%%%%%%%%%%%%%%%%%%%%%%%%%%%%%%%%%%%%%%%%%%%%%%%%%%%%%%%%%


\begin{thebibliography}{10}

\bibitem{amanba26}
T. Amdeberhan, G.~E. Andrews and C.~M. Ballantine, Lambert series and double Lambert series, J. Combin. Theory Ser. A {\bf 221} (2026), Paper No. 106154, 22 pp.; MR5000708

\bibitem{adsy17}
G.~E. Andrews, A. Dixit, D. Schultz, and A.~J. Yee, Overpartitions related to the mock theta function $\omega(q)$, Acta Arith. {\bf 181} (2017), no.~3, 253--286; MR3732919

\bibitem{ady15}
G.~E. Andrews, A. Dixit and A.~J. Yee, Partitions associated with the Ramanujan/Watson mock theta functions $\omega(q)$, $\nu(q)$ and $\phi(q)$, Res. Number Theory {\bf 1} (2015), Paper No. 19, 25 pp.; MR3501003

\bibitem{be91}
B.~C. Berndt, {\it Ramanujan's notebooks. Part III}, Springer, New York, 1991; MR1117903

\bibitem{ct26}
S. P. Cui and D. Tang, Identities and transformations for Lambert series and double Lambert series, preprint, https://arxiv.org/pdf/2604.08839, 2026.

\bibitem{fang26}
Q. Fang, On the double Lambert series conjecture of Andrews-Dixit-Schultz-Yee, preprint, https://arxiv.org/pdf/2604.06242, 2026.

\bibitem{red}
G. Gasper and M. Rahman, {\it Basic hypergeometric series}, second edition, Encyclopedia of Mathematics and its Applications, 96, Cambridge Univ. Press, Cambridge, 2004; MR2128719

\bibitem{ks26}
R. Kumar and A. Singh, On a conjecture of Amdeberhan, Andrews and Ballantine for double Lambert series, preprint, https://doi.org/10.48550/arXiv.2605.21163, 2026.

\bibitem{lam65}
J.~H. Lambert, Anlage zur Architectonic, oder Theorie des ersten und des einfachen in der philosophischen und mathematischen Erkenntnis, Vol. 2, Johann Friedrich Hartenoch, Riga, 1771, Philosophische Schriften, vol. 4, Georg Olm, Hildesheim, 1965.

\bibitem{sch20}
M. D. Schmidt, A catalog of interesting and useful Lambert series identities, preprint, https://doi.org/10.48550/arXiv.2004.02976, 2020.


\end{thebibliography}
\end{document}